\documentclass[12pt]{amsart}
\title{Building trees in large fields}
\author{Elliot Kaplan}
\address{Universit\'{e} de Mons, D\'{e}partement de Math\'{e}matique, Belgium}
\email{elliot.kaplan@umons.ac.be}

\author{Angus Matthews}
\address{School of Mathematics, University of Leeds, United Kingdom}
\email{mmamat@leeds.ac.uk}

\author{Erik Walsberg}
\address{Universit\"{a}t Wien, Institut f\"{u}r Mathematik, Kurt G\"{o}del Research Center, Austria}
\email{erik.walsberg@gmail.com}

\usepackage{amsmath, amssymb, amsthm, enumitem, comment}    
\usepackage[mathscr]{euscript}
\usepackage[Symbol]{upgreek}
\usepackage{fullpage} 	
\usepackage{hyperref}
\usepackage{lineno}
\usepackage[all]{xy}
\usepackage{centernot}
\usepackage{tikz-cd}
\usepackage{xspace}
\usepackage{enumitem}
\usepackage{mathtools}
\usepackage{cleveref}
\usepackage[T1]{fontenc}

\DeclareFontFamily{U}{fsy}{}
\DeclareFontShape{U}{fsy}{m}{n}{<->s*[.9]psyr}{}
\DeclareSymbolFont{der@m}{U}{fsy}{m}{n}
\DeclareMathSymbol{\der}{\mathord}{der@m}{182}

\DeclareFontFamily{U}{BOONDOX-calo}{\skewchar\font=45 }
\DeclareFontShape{U}{BOONDOX-calo}{m}{n}{
  <-> s*[1.05] BOONDOX-r-calo}{}
\DeclareFontShape{U}{BOONDOX-calo}{b}{n}{
  <-> s*[1.05] BOONDOX-b-calo}{}
\DeclareMathAlphabet{\mathcalboondox}{U}{BOONDOX-calo}{m}{n}
\SetMathAlphabet{\mathcalboondox}{bold}{U}{BOONDOX-calo}{b}{n}
\DeclareMathAlphabet{\mathbcalboondox}{U}{BOONDOX-calo}{b}{n}

\DeclareSymbolFont{imag@m}{OT1}{cmr}{m}{ui}
\DeclareMathSymbol{\imag}{\mathord}{imag@m}{105}

\DeclareMathOperator*{\forkindep}{\raise0.2ex\hbox{\ooalign{\hidewidth$\vert$\hidewidth\cr\raise-0.9ex\hbox{$\smile$}}}}

\newcommand{\Sa}[1]{\ensuremath{\mathscr{#1}}}

\newcommand{\aff}{\operatorname{Aff}}

\newcommand{\Gal}{\operatorname{Gal}}
\newcommand{\Spec}{\operatorname{Spec}}

\newcommand{\Aut}{\operatorname{Aut}}

\newcommand{\tp}{\operatorname{tp}}

\newcommand{\Chara}{\operatorname{Char}}

\newcommand{\prc}{\mathrm{PRC}}
\newcommand{\pac}{\mathrm{PAC}}

\newcommand{\cY}{\mathcal{Y}}
\newcommand{\cZ}{\mathcal{Z}}

\newtheorem*{fact*}{Fact}
\newtheorem*{Claim*}{Claim}
\newtheorem*{theorem*}{Theorem}
\newtheorem{theorem}{Theorem}[section] 
\newtheorem{lemma}[theorem]{Lemma}

\newtheorem{prop-def}[theorem]{Proposition-Definition}
\newtheorem{corollary}[theorem]{Corollary}
\newtheorem{fact}[theorem]{Fact}
\newtheorem{fact-eh}[theorem]{Fact(?)}

\newtheorem{proposition-eh}[theorem]{Proposition(?)}
\newtheorem*{theorem-star}{Theorem}
\newtheorem*{conjecture-star}{Conjecture}
\newtheorem*{lemma-star}{Lemma}

\newtheorem{thmA}{Theorem}

\newtheorem{corA}{Corollary}

\newtheorem*{lemma6.2alt}{Lemma \ref*{confusing-v2}$'$}

\theoremstyle{definition}

\theoremstyle{remark}

\newcommand{\Aa}{\mathbb{A}}

\newcommand{\Qq}{\mathbb{Q}}

\newcommand{\Zz}{\mathbb{Z}}
\newcommand{\Nn}{\mathbb{N}}

\DeclareTextSymbol{\thh}{T1}{254}
\def\th{\textnormal{\thh}}

\newenvironment{claimproof}[1][\proofname]
               {
                 \proof[#1]
                 
               }
               {
                 \endproof
               }

\begin{document}

\begin{abstract}
We show that large rosy fields are bounded, and substantially simplify the proofs that large stable fields are separably closed and that large simple fields are bounded.
Our proofs go through in a general topological setting.
We obtain instability, non-simplicity, and non-rosiness explicitly by building the appropriate trees of definable sets.
We also show that orders on large rosy fields have several properties of orders on pseudo real closed fields.
\end{abstract}

\maketitle

\section{Introduction} 
Throughout $K$ is a field.
Recall that $K$ is {\bf large} if any $K$-curve with a smooth $K$-point has infinitely many $K$-points, $K$ is {\bf virtually large} if some finite extension of $K$ is large, and $K$ is {\bf bounded} if $K$ has only finitely many separable extensions of any fixed finite degree.
Equivalently $K$ is bounded if the absolute Galois group of $K$ depends only on the first-order theory of $K$.
All known logically tame infinite fields are virtually large and almost all of them are large.
Several intractable conjectures from the model theory of fields have been proven for large or virtually large fields.
In particular, virtually large stable fields are separably closed~\cite{firstpaper} and large fields with simple theory are bounded~\cite{with-anand}.
We simplify the proof of the first result and generalize the second to virtually large rosy fields.
The model-theoretic part of the proofs goes through in a more general setting.

\begin{thmA}\label{thm:elastic}
Suppose that $\Sa M$ is a first-order structure equipped with a topology $\uptau$ on some $M^n$.
Let $\Gamma$ be a $2$-elastic definable family of homeomorphisms $M^n \to M^n$.
\begin{enumerate}[leftmargin=*]
\item\label{B_stable} If $\Sa M$ is stable  then any definable subset of $M^n$ with interior is dense.
\end{enumerate}
Suppose in addition that $\Sa M$ is sufficiently saturated, $\Gamma$ is elastic, and $\uptau$ is $\Aut(\Sa M)$-invariant.
\begin{enumerate}[leftmargin=*]
\setcounter{enumi}{1}
\item\label{B_rosy} If $\Sa M$ is rosy then no definable subset of $M^n$ with interior \textup{\thh}-divides.
\item\label{B_simple} If $\Sa M$ is simple then no definable subset of $M^n$ with interior divides.
\end{enumerate}
\end{thmA}

We say that such $\uptau$ is {\bf $\Aut(\Sa M)$-invariant} if for any $\sigma \in \Aut(\Sa M)$, the map $\sigma^n \colon M^n \to M^n$ given by applying $\sigma$ coordinate-wise is a homeomorphism.
Let $k$ be a positive integer and $X$ be a topological space.
A family $\Gamma$ of homeomorphisms $X \to X$ is {\bf $k$-elastic} if for any nonempty open sets $U, V_1, \ldots, V_k\subseteq X$, there is $\gamma \in \Gamma$ such that each $\gamma(V_i)$ intersects $U$.
It is \textbf{elastic} if it is $k$-elastic for each $k \ge 1$.
The notion of $k$-elasticity was introduced for group actions in~\cite{CKN07} as a weakening of topological $k$-transitivity.
Note that $1$-elasticity is just topological transitivity.

\medskip
Let $\uptau$ be a topology on $K^n$.
Let $\aff_n(K)$ be the group of invertible affine transformations of $K^n$.
We say that $\uptau$ is \textbf{affine} if every $h \in \aff_n(K)$ is a homeomorphism.
We will see that if $K$ is infinite and $\uptau$ is affine and non-discrete then the action of $\aff_n(K)$ on $K^n$ is elastic.
Of course $2$-elasticity follows from $2$-transitivity, but full elasticity requires a short argument, see Lemma~\ref{lem:affine-stuffing}.
Thus Theorem~\ref{thm:elastic} yields the following, where $\Sa K$ is an expansion of $K$.

\begin{thmA}\label{thm:fields}
Suppose that $\uptau$ is a non-discrete affine topology on $K^n$.
\begin{enumerate}[leftmargin=*]
\item If $\Sa K$ is stable then any definable subset of $K^n$ with interior is dense.
\end{enumerate}
Suppose in addition that $\Sa K$ is sufficiently saturated and $\uptau$ is $\Aut(\Sa K)$-invariant.
\begin{enumerate}[leftmargin=*]
\setcounter{enumi}{1}  
\item\label{fields_rosy} If $\Sa K$ is rosy  then a definable subset of $K^n$ with interior cannot \textup{\thh}-divide.
\item If $\Sa K$ is simple then a definable subset of $K^n$ with interior cannot divide.
\end{enumerate}
\end{thmA}

Of course, this is only useful if we have non-discrete invariant affine topologies.
We consider the \'etale-open (or $\Sa E_K$-) topology on $K^n$ introduced in~\cite{firstpaper}.
See Section~\ref{section:B} for the definition.
The $\Sa E_K$-topology on $K^n$ is $K$-invariant by definition, is discrete if and only if $K$ is not large, and is affine as any polynomial map is continuous.
The $\Sa E_K$-topology admits a basis of definable sets and it is shown in~\cite{firstpaper} that it is not Hausdorff if and only if $K$ is  separably closed if and only it agrees with the Zariski topology.
Note that this gives a disjoint pair of nonempty definable open subsets of $K$ when $K$ is not separably closed.
Thus, it follows from Theorem~\ref{thm:elastic} that  non-separably closed large fields are unstable.
From this, it is easy to show that non-separably closed virtually large fields are unstable as well.
Furthermore, it is shown in~\cite{large->henselian} that if $f \in K[x]$ is separable, irreducible, and of degree at least $2$, then $f(K)$ and $K \setminus f(K)$ both have interior.
Combining with the proof of Theorem~\ref{thm:elastic}, we get an explicit unstable formula.

\begin{corA}\label{cor:unstable}
If $K$ is a large field and $f \in K[y]$ is  separable, irreducible, and of degree at least $2$, then the following formula $\varphi(x;s,t)$ is unstable $$\exists y [f(y) = sx + t].$$
\end{corA}

Note this is sharp in that any quantifier-free formula in a field is stable.
We now consider boundedness.
It is shown in~\cite[Section~4]{with-anand} that for any $n \ge 2$ there is a definable $\Sa E_K$-open subset $U$ of $K^n$ and a definable equivalence relation $E$ on $U$ with $\Sa E_K$-open classes such that $E$-classes correspond to separable extensions of degree $n$.
It therefore follows from Theorem~\ref{thm:fields}\eqref{fields_rosy} that large rosy fields are bounded.
(Recall that simple structures are rosy.)
The generalization to virtually large fields is again immediate.

\begin{corA}\label{cor:rosy_bounded}
Any virtually large rosy field is bounded.
\end{corA}

Our original proof of Corollary~\ref{cor:rosy_bounded} (and a more restrictive version of Theorem~\ref{thm:elastic}) followed the approach by Fornasiero and the first two authors in~\cite{FKM26}.
The argument we give here is significantly simpler. 
Our theorems are instances of the philosophy that results about large fields should be special cases of results about fields equipped with topologies, see~\cite{topological_proofs} for other instances of this.

\medskip
Bounded pseudo real closed fields are rosy; in particular, bounded $\pac$ fields such as pseudofinite fields are rosy~\cite{Cha-Pi,Mo17B}.
It is conjectured that any rosy field is bounded and pseudo real closed ($\prc$)~\cite{Kr15}.
We show that field orders on large rosy fields satisfy several properties of field orders on bounded $\prc$ fields~\cite{Samaria-2017,prestel-prc}.

\begin{thmA}\label{thm:field-orders}
Suppose that $K$ is large and rosy.
Then we have the following.
\begin{enumerate}[leftmargin=*]
\item $K$ admits only finitely many field orders, each of which is definable.
\item If $<$ is a field order on $K$ then $K$ is dense in the real closure of $(K, <)$.
\item If $<_1, <_2$ are distinct field orders on $K$ then any nonempty open $<_1$-interval intersects any nonempty open $<_2$-interval.
\end{enumerate}
\end{thmA}

Here (1) is a consequence of boundedness while (2) and (3) each require another application of rosiness.
The third item follows from a more general result showing that if two distinct field orders induce the same topology then both orders are ``definably non-archimedean". Note that by (2), any non-trivial valuation $v$ on $K$ that is convex with respect to some field order on $K$ has divisible value group and real closed residue field. Using a result of Krupi\'nski~\cite{Kr15}, we can show that any valuation on $K$ has divisible value group, assuming $K$ is perfect; see Corollary~\ref{cor:k} below.

\medskip
Theorem~\ref{thm:elastic}\eqref{B_stable} is sharp in the sense that 1-elasticity (even transitivity) is not enough.
To see this, consider $\Zz$ as an additive group, equipped with the topology with basic open sets of the form $n\Zz + \beta$ for $n \ge 1$ and $\beta \in \Zz$.
This topology is non-discrete and there are obviously pairs of disjoint nonempty definable open sets.
The natural action of $\Zz$ on itself is a definable action by homeomorphisms and $\Zz$ is superstable.
Theorem~\ref{thm:fields} appears sharp in that it does not go through for other model-theoretic tree properties.
There are $\mathrm{NTP}_2$ fields such as $\Qq_p$ in which definable open sets such as $\Zz_p$ divide.
There are also unbounded $\mathrm{NSOP}_1$ fields that are $\pac$ and hence large~\cite[Cor.~6.2]{cr-tree}.
The proof of Corollary~\ref{cor:rosy_bounded} shows that such fields define $\Sa E_K$-open subsets of $K^n$ which divide.

\subsection*{Conventions}\label{section:conventions}
Throughout $n, m, k$ are natural numbers (including $0$).
We let $\Chara(K)$ be the characteristic of the field $K$.
All structures are first-order and ``definable'' without modification means ``first-order definable, possibly with parameters''.

\subsection*{Acknowledgments}
The first author is supported by a CR fellowship through the Fonds de la Recherche Scientifique -- FNRS.
The second author's work is supported by EPSRC DTP 2224 [EP/W524372/1] (at the University of Leeds).
The third author is supported by the Austrian Science Fund (FWF) 10.55776/PAT1673125.
The authors thank the organizers of the first annual Chania model theory conference, at which this work was begun.
We also thank Philip Dittmann for explaining how to prove Fact~\ref{fact:r-bndd} and Aristotelis Panagiotopoulos for pointing us to the notion of elasticity.

\section{Proof of Theorem~\ref{thm:elastic}}\label{section:A}

In proving Theorem~\ref{thm:elastic}\eqref{B_stable}, we use the following standard fact~\cite[Thm.~8.2.3]{TZ12}.

\begin{fact}\label{fact:stable}
Let $\Sa M$ be a  structure and $\delta(x; y)$ be a formula.
Suppose that some collection of boolean combinations of members of $\{ \delta(M^{|x|}; b) : b \in M^{|y|} \}$ forms a complete binary tree under inclusion.
Then $\delta(x; y)$ is not stable.
\end{fact}

\begin{proof}[Proof of Theorem~\ref{thm:elastic}\eqref{B_stable}]
Suppose that $X$ is a definable subset of $M$ such that $X$ and $M \setminus X$ both have interior. Let $\Gamma(X) \coloneqq (\gamma(X))_{\gamma \in \Gamma}$.
Let $\delta(x; \gamma)$ be the formula defining $\Gamma(X)$, i.e. $\exists y [y \in X] \land [\gamma(y) = x]$.
We show that $\delta(x ; \gamma)$ is unstable by applying Fact~\ref{fact:stable}.
It suffices to prove that if $Z$ is an intersection of sets in $\Gamma(X)$ and complements of sets in $\Gamma(X)$ such that $Z$ has interior, then there is $\gamma \in \Gamma$ such that $\gamma(X)\cap Z$ and $\gamma(M^n \setminus X) \cap Z = Z \setminus \gamma(X)$ both have interior.
This is immediate by $2$-elasticity.
\end{proof}

For the proof of Theorem~\ref{thm:elastic}\eqref{B_rosy}, we recall the definition of \textup{\thh}-rank  and \textup{\thh}-dividing from~\cite{On06, EO07}.
Suppose that $\Sa M$ is sufficiently saturated.
Let $\cY = (Y_z)_{z\in Z}$ be a definable family of subsets of $M^n$, indexed by a definable set in $\Sa M^{\operatorname{eq}}$.
A member $Y_{z_0}$ of this family is said to \textbf{\textup{\thh}-divide} if there is a parameter set $A$ extending the defining parameters for $\cY$ such that the type $\tp(z_0/A)$ is non-algebraic and the family $(Y_z)_{z \models \tp(z_0/A)}$ is $k$-inconsistent. We note that \textup{\thh}-dividing depends not just on the set $Y_{z_0}$, but also on the family $\cY$ from which it comes. 

\medskip
Now consider also a definable subset $X$ of $M^n$, a definable family $\cZ = (Z_w)_{w \in W}$ of subsets of $Z$, and $k \ge 1$.
We define the \textbf{\textup{\thh}-rank} $\th(X,\cY,\cZ,k)$ as follows:

\begin{enumerate}[label=(\roman*)]
\item $\th(X,\cY,\cZ,k)\geq 0$ if $X$ is nonempty, and
\item $\th(X,\cY,\cZ,k)\geq n+1$ if there is $Z_w \in \cZ$ so that for every $m \ge 1$ there are distinct elements $z_1, \ldots, z_m$ of $Z_w$ such that
\begin{enumerate}
\item the family $(Y_z)_{z \in Z_w}$ is $k$-inconsistent, and 
\item $\th(X \cap Y_{z_i},\cY,\cZ,k)\geq n$ for each $i = 1, \ldots, m$.
\end{enumerate}
\item $\th(X,\cY,\cZ,k)$ is the supremum in $\Nn \cup \{ \pm \infty \}$ of the set of $n$ such that $\th(X, \cY, \cZ, k) \ge n$.
\end{enumerate}

If $\th(X,\cY,\cZ,k)$ is finite for all $X,\cY,\cZ,k$, then $\Sa M$ is said to be \textbf{rosy}. 
The usual definition of rosiness is slightly different.
Instead of asking for arbitrarily long finite sequences $z_1, \ldots, z_m$, one requires an infinite sequence $(z_i)_{i < \omega}$ of distinct elements of $Z_w$ satisfying (b).
A routine saturation argument shows that our definition is equivalent to the usual definition for sufficiently saturated $\Sa M$.

\begin{proof}[Proof of Theorem~\ref{thm:elastic}(\ref{B_rosy}, \ref{B_simple})]
We first prove~\eqref{B_rosy}. 
Let $\cY = (Y_z)_{z \in Z}$ be a definable family of subsets of $M^n$ and suppose that some $Y_{z_0} \in \cY$ has interior and \thh-divides. By definition, there is a small parameter set $A$ extending the defining parameters for $Y$ such that $\tp(z_0/A)$ is nonalgebraic and $(Y_z)_{z \models \tp(z_0/A)}$ is $k$-inconsistent for some $k \ge 1$.
Saturation yields a (necessarily infinite) $A$-definable set $Z_0$ containing $z_0$ such that $(Y_z)_{z \in Z_0}$ is $k$-inconsistent. Replacing $Z$ with $Z_0$, we may assume that our original family $\cY$ is $k$-inconsistent.
Set 
\[
\Gamma(\cY) \coloneqq (\gamma(Y_z))_{(\gamma,z) \in \Gamma\times Z} \qquad \cZ \coloneqq (\{\gamma \}\times Z)_{\gamma \in \Gamma}.
\]
We show that $\th(M^n,\Gamma(\cY),\cZ,k)$ is infinite; it follows that $\Sa M$ is not rosy. 
Note that the family $(\gamma(Y_z))_{z \in Z}$ is $k$-inconsistent for any $\gamma \in \Gamma$.
By a straightforward inductive argument, it suffices to show that if $X$ is a finite intersection of sets in $\Gamma(\cY)$ and $X$ has interior, then for any $m \ge 1$ we can find $\gamma \in \Gamma$ and a sequence $z_1, \ldots, z_m$ of distinct elements of $Z$ such that each $\gamma(Y_{z_i})\cap X$ has interior.
Using invariance of the topology and saturation, we find distinct $z_1, \ldots, z_m \in Z$ such that each $Y_{z_i}$ has interior.
Then elasticity yields the desired $\gamma$. 

\medskip
Part~\eqref{B_simple} can be established with obvious changes to the proof above. Indeed, the proof shows that if some $Y_{z_0} \in \cY$ with interior divides, then the family $\Gamma(\cY)$ has the tree property. 
\end{proof}
It is an open question whether \thh-dividing agrees with dividing over simple theories, so part~\eqref{B_simple} of Theorem~\ref{thm:elastic} may not follow on general grounds from part~\eqref{B_rosy}.

\begin{corollary}\label{cor:eq}
Let $\Sa M$, $\uptau$, and $\Gamma$ be as in the second part of Theorem~\ref{thm:elastic}. 
If  $\Sa M$ is rosy and $E$ is a definable equivalence relation on a definable subset $U\subseteq M^n$  with open classes, then $U/E$ is finite.
\end{corollary}

\begin{proof}
If there were infinitely many $E$-classes, then some member of the definable family of $E$-classes (parametrized by $U/E$) would \thh-divide.
This contradicts Theorem~\ref{thm:elastic}\eqref{B_rosy}.
\end{proof}

\section{Consequences over fields}\label{section:B}

Throughout this section, $\Sa K$ is an expansion of an infinite field $K$, $n \ge 1$, and $\uptau$ is a topology on $K^n$.
Recall that $\uptau$ is {\bf affine} if any invertible affine map $K^n \to K^n$ is a homeomorphism.
\begin{lemma}\label{lem:affine}
Suppose that $\uptau$ is affine and non-discrete.
Let $U$ be a non-empty open subset of $K^n$.
Then $U$ is not contained in a finite union of proper affine subspaces of $K^n$.
\end{lemma}

When $K$ is large and $\uptau$ is the $\Sa E_K$-topology, Lemma~\ref{lem:affine} follows as any nonempty open subset of $V(K)$ is Zariski dense in $V$ when $V$ is a smooth irreducible $K$-variety~\cite[Prop.~5]{topological_proofs}.

\begin{proof}[Proof of Lemma~\ref{lem:affine}]
Suppose otherwise and let $X$ be a finite union of proper affine subspaces of $K^n$ containing $U$.
After possibly translating, suppose that $U$ contains the origin.
By elementary linear algebra there are $h_1, \ldots, h_m \in \mathrm{Gl}_n(K)$ with $h_1(X) \cap \cdots \cap h_m(X) = \{ 0 \}$.
Then we have $\{ 0 \} = h_1(U) \cap \cdots \cap h_m(U)$, so $\{ 0 \}$ is open, contradicting non-discreteness.
\end{proof}

\begin{lemma}\label{lem:affine-stuffing}
Let $\uptau$ be affine and non-discrete.
Then the action of $\aff_n(K)$ on $K^n$ is elastic.
\end{lemma}

\begin{proof}
Let $U, V_1, \ldots, V_k$ be nonempty open subsets of $K^n$.
We produce $h \in \aff_n(K)$ such that each $h(V_i)$ intersects $U$.
After possibly translating suppose $U$ contains the origin.
Let $D = (K^\times)^n$, note that $K^n \setminus D$ is a finite union of proper affine subspaces.
Identify $D$ with the diagonal subgroup of $\mathrm{Gl}_n(K)$.
Applying Lemma~\ref{lem:affine} fix $g_i \in V_i \cap D$ for each $i$.
Let $U^* = \bigcap_{i = 1}^{k} g^{-1}_i(U)$, so $U^*$ is a neighborhood of the origin.
By Lemma~\ref{lem:affine} there is $h \in U^* \cap D$.
Then $h$ is in each $g^{-1}_i(U)$. For each $i$, the maps $g_i(\cdot)$ and $h(\cdot)$ both correspond to diagonal matrices. Therefore $h(g_i) = g_i(h) \in U$, hence each $h(V_i)$ intersects $U$.
\end{proof}

We now recall the \'etale-open topology from~\cite{firstpaper}.
Let $V, W$ range over $K$-varieties, i.e. schemes of finite type over $K$.
(For our purposes we could also restrict to quasi-projective or even affine varieties.)
Let $V(K)$ be the set of $K$-points of $V$.
Let $\Aa^n$ be $n$-dimensional affine space over $K$, so $\Aa^n = \Spec K[x_1, \ldots, x_n]$ and $\Aa^n(K) = K^n$.
The {\bf \'etale-open topology}, or $\Sa E_K$-topology, on $W(K)$ is the topology with open basis the collection of sets of the form $f(V(K))$ for $f$ an \'etale morphism $V \to W$.
In particular, a subset of $K^n$ is open when it is a union of sets of the form $f(V(K))$ for $f \colon V \to \Aa^n$ \'etale.
We recall some basic facts.
\begin{enumerate}[leftmargin=*]
\item $K$ is large if and only if the $\Sa E_K$-topology on $K$ is not discrete if and only if the $\Sa E_K$-topology on $V(K)$ is not discrete when $V(K)$ is infinite~\cite[Thm.~7.1]{firstpaper}.
\item If $V \to W$ is a morphism of $K$-varieties then the induced map $V(K) \to W(K)$ is continuous with respect to the $\Sa E_K$-topology~\cite[Lemma~5.2]{firstpaper}.
In particular, the $\Sa E_K$-topology on $K^n$ is affine.
\item The $\Sa E_K$-topology on $K^n$ is $\Aut(K)$-invariant.
This follows immediately from the definition, see also~\cite[Prop.~2.4]{secondpaper}.
\end{enumerate}

Hence if $K$ is large and sufficiently saturated then we can apply Theorem~\ref{thm:fields} to the $\Sa E_K$-topology on $K^n$.

\begin{proof}[Proof of Corollary~\ref{cor:unstable}]
Let $K$ be a non-separably closed large field, equipped with the $\Sa E_K$-topology.
We describe an explicit unstable formula.
It follows by the proof of Theorem~\ref{thm:elastic} that if $X$ is a definable subset of $K$ such that $X$ and $K \setminus X$ both have interior then the formula $\delta(x ; s, t)$ given by $\exists y [y \in X] \land [y = sx + t]$ is unstable.
Let $f \in K[y]$ be separable, irreducible, and of degree at least $2$.
Then the morphism $\Aa^1 \to \Aa^1$ given by $f$ is \'etale on the subvariety of $\Aa^1$ given by $f' \ne 0$.
Hence $f(K)$ has interior.
Furthermore, $0$ is in the interior of $K \setminus f(K)$, see the proof of~\cite[Prop.~4.6]{large->henselian}.
Hence the formula $\varphi(x ; s, t)$ given by
$$ \exists y [f(y) = sx + t]$$
is unstable.
\end{proof}

\begin{proof}[Proof of Corollary~\ref{cor:rosy_bounded}]
Let $K$ be rosy and virtually large. We will show that $K$ is bounded. Let $F$ be a large finite extension of $K$, so $F$ is also rosy, as it is interpretable in $K$.
By Fact~\ref{fact:bndd} below it suffices to show that $F$ is bounded.
Hence we may suppose that $K$ is large.
Let $K^*$ be a sufficiently saturated elementary extension of $K$.
Then $K^*$ is rosy and $K^*$ is bounded if and only if $K$ is bounded.
Hence we may also suppose that $K$ is sufficiently saturated.
Fix $n \ge 2$.
We show that $K$ has only finitely many separable extensions of degree $n$.
Equip $K^n$ with the \'etale-open topology.
Given $a = (a_0, \ldots, a_{n - 1}) \in K^n$ let $p_a(x) = x^n + a_{n - 1} x^{n - 1} + \cdots + a_1 x + a_0 \in K[x]$.
Let $U$ be the set of $a \in K^n$ such that $p_a$ is separable and irreducible and let $E$ be the equivalence relation on $U$ given by declaring $E(a, b)$ when $K[x]/(p_a)$ is isomorphic as a $K$-algebra to $K[x]/(p_b)$.
It is enough to show that there are only finitely many $E$-classes. 
Note that $U$ and $E$ are definable.
By Corollary~\ref{cor:eq} it is enough to show that every $E$-class is an open subset of $K^n$.
This is~\cite[Thm.~4.3]{with-anand}.
\end{proof}

\begin{fact}\label{fact:bndd}
If some finite extension of $K$ is bounded then $K$ is bounded.
\end{fact}

\begin{proof}
Fix some $n$, let $\{K_s : s \in S\}$ be the set of degree $n$ separable extensions of $K$, and let $F$ be a finite bounded extension of $K$.
By boundedness, $\{FK_s : s \in S\}$ is finite, say equal to $\{F_1, \ldots, F_m\}$.
Every $K_s$ is contained in some $F_i$.
Any finite field extension contains only finitely many separable subextensions.
\end{proof}

Krupi\'nski proved Corollary~\ref{cor:k} for superrosy fields~\cite{Kr15}.
With Corollary~\ref{cor:rosy_bounded} in place it goes through for large rosy fields.

\begin{corollary}\label{cor:k}
Suppose that $K$ is large, rosy, and perfect.
Then any non-trivial valuation on $K$ has divisible value group.
If $K$ additionally has positive characteristic then any non-trivial valuation on $K$ has algebraically closed residue field.
\end{corollary}

Corollary~\ref{cor:k} follows from Facts~\ref{fact:r-bndd} and~\ref{fact:k}, and the trivial fact that a rosy field does not admit a non-trivial definable valuation.
When $\Chara(K) = p$ we let $\wp(x) = x^p - x$.
Recall that $K$ is {\bf radically bounded} if any finite extension $F/K$ satisfies the following:
\begin{enumerate}[leftmargin=*]
\item The group of nonzero $n$th powers has finite index in $F^\times$ for any $n \ge 2$,
\item If $\Chara(K)$ is positive then $\wp(F)$ is a finite index subgroup of the additive group of $F$.
\end{enumerate}

\begin{fact}\label{fact:r-bndd}
A perfect bounded field is radically bounded.
\end{fact}

This is well-known but we could not find a reference.
A proof of the characteristic zero case is given in~\cite{FJ-bounded}.
We sketch a proof for the sake of completeness.

\begin{proof}
Bounded fields are closed under finite extensions, so it is enough show that (1) and (2) above hold in the case $F = K$.
Fix $n \ge 2$.
Let $\mu_n$ be the group of $n$th roots of unity in the algebraic closure of $K$ and let $P$ be the subgroup of $n$th powers in $K^\times$.
As $K$ is perfect we may suppose that $n$ is coprime to $\Chara(K)$ when $\Chara(K)$ is positive.
Let $\Gal_K$ be the absolute Galois group of $K$.
It is well-known that $H^1(\Gal_K, \mu_n)$ is isomorphic to $K^\times/P$ and $H^1(\Gal_K, \Zz/p\Zz)$ is isomorphic to $K/\wp(K)$ when $\Chara(K) = p$.
Both of these cohomology groups are finite when $K$ is bounded~\cite[III, \S 4]{Serre_gcoho}.
\end{proof}

Fact~\ref{fact:k} is a theorem of Krupi\'nski~\cite{Kr15}.

\begin{fact}\label{fact:k}
Suppose that $K$ is radically bounded.
Then one of the following holds.
\begin{enumerate}[leftmargin=*]
\item There is a non-trivial definable valuation on $K$.
\item If $v$ is a non-trivial valuation on $K$ then the value group of $v$ is divisible and if $K$ additionally has positive characteristic, then the residue field of $K$ is algebraically closed.
\end{enumerate}
\end{fact}

\begin{proof}[Proof of Theorem~\ref{thm:field-orders}]
We may suppose that $\Chara(K) = 0$.
We need to prove the following.
\begin{enumerate}[leftmargin=*]
\item $K$ admits only finitely many field orders, each of which is definable.
\item If $<$ is a field order on $K$ then $K$ is dense in the real closure of $(K, <)$.
\item Any two open intervals with respect to distinct field orders intersect.
\end{enumerate}
Let $S$ be the set of non-zero squares in $K$.
By Corollary~\ref{cor:rosy_bounded} and Fact~\ref{fact:r-bndd} $K^\times/S$ is finite.
It follows that if $<$ is a field order on $K$ then $\{a \in K : a > 0\}$ is a finite union of cosets of $S$.
Note that (1) follows.
If $K$ is not dense in the real closure of $(K, <)$ then $(K, <)$ has a non-trivial definable valuation~\cite[Prop.~6.5]{JSW-field}, contradicting (1) and rosiness.
For (3), it suffices to show that distinct field orders on $K$ induce distinct topologies; 
see~\cite[Thm.~4.1]{Prestel1978}. This follows from (1) and Lemma~\ref{lem:rosy orders} below.
\end{proof}

Lemma~\ref{lem:rosy orders} is a definable version of the easy fact that if $<$ is an archimedean field order on $K$ then any field order on $K$ which induces the same topology as $<$ must agree with $<$.

\begin{lemma}\label{lem:rosy orders}
Let $K$ be a characteristic zero field and $<_1, <_2$ be distinct field orders on $K$ which induce the same topology.
Then there is a non-trivial $(K, <_1, <_2)$-definable $<_i$-convex valuation for $i = 1, 2$.
\end{lemma}

Let $(K, <)$ be an ordered field.
A valuation on $(K, <)$ is {\bf convex} if its valuation ring is convex.
Any non-trivial convex subring of $K$ is a valuation ring.
Our proof is motivated by the proof that the $\Sa E_K$-topology on $K$ is definably connected if and only if $K$ is  real or separably closed~\cite[Prop.~7.15]{firstpaper}.

\begin{proof}
We work in $(K, <_1, <_2)$ with respect to the topology induced by either $<_i$.
To simplify notation $<$ is $<_1$ and intervals/convex sets/absolute values are with respect to $<_1$.
If $H$ is a non-trivial definable convex additive subgroup then it is easy to see that $\{a \in K : aH \subseteq H \}$ is a non-trivial definable convex subring.
If $J$ is a non-trivial definable open additive subgroup then it is also easy to see that $\{ b \in K : [-b, b] \subseteq J\}$ is a non-trivial definable convex additive subgroup.
Hence it suffices to produce a non-trivial definable open additive subgroup.

\medskip
Let $\Gamma$ be the set of non-zero elements of $K$ which have the same sign with respect to $<_1$ and $<_2$.
Then $\Gamma$ is a definable open subgroup of $K^\times$ which contains $\pm 1$.
We have $\Gamma \ne K^\times$ as $<_1, <_2$ are distinct.
We only need these properties moving forwards.
Let $P = \Gamma \cup (1 + \Gamma)$.
Let $G$ be the additive stabilizer of $P$, i.e. the set of $b \in K$ such that $P + b = P$.
Then $G$ is a definable additive subgroup of $K$.
We show that $G$ is open and $G \ne K$.

\begin{Claim*}
Let $a, b$ range over $K^\times$.
There is a constant $c > 0$ such that if $a, b$ lie in distinct cosets of $\Gamma$ then $|a|, |b|$ are both bounded above by $c |a - b|$.
\end{Claim*}

\begin{claimproof}

Fix $\delta > 0$ such that $[1 - \delta, 1 + \delta] \subseteq \Gamma$.
Then $a [1 - \delta, 1 + \delta] = [a - \delta |a|, a + \delta |a| ]$ is contained in the same coset of $\Gamma$ as $a$.
Hence if $a,b$ are in distinct cosets of $\Gamma$ then $|a - b| > \delta |a|$, so $\delta^{-1} |a-b| > |a|$.
Take $c = \delta^{-1}$.
\end{claimproof}

We show that $G \ne K$.
Equivalently: $P \ne K$.
If $\Gamma$ is bounded then $P$ is bounded, hence $P \ne K$.
Suppose that $\Gamma$ is unbounded.
The claim shows that if $a, a - 1$ are in distinct cosets of $\Gamma$ then $|a| \le c$.
Hence $\Gamma \setminus [-c, c] = (1 + \Gamma) \setminus [-c, c] = P \setminus [-c, c]$.
So it is enough to show that $K \setminus \Gamma$ intersects $K \setminus [-c, c]$.
This follows as every coset of $\Gamma$ is unbounded.

\medskip
It remains to show that $G$ is open.
It suffices to show that all sufficiently small elements of $K$ are in $G$.
Suppose otherwise.
For every $\varepsilon > 0$ there are  $a \in P, b \in K^\times \setminus P$ and $|a - b| < \varepsilon$.
Hence one of the following holds.
\begin{enumerate}[leftmargin=*]
\item For every $\varepsilon > 0$ there are $a_\varepsilon \in \Gamma, b_\varepsilon \in K^\times \setminus P$ such that $|a_\varepsilon - b_\varepsilon| < \varepsilon$.
\item For every $\varepsilon > 0$ there are $a_\varepsilon \in 1 + \Gamma, b_\varepsilon \in K^\times \setminus P$ such that $|a_\varepsilon - b_\varepsilon| < \varepsilon$.
\end{enumerate}
Suppose (1).
By the claim $|a_\varepsilon|, |b_\varepsilon|$ are both bounded above by $c | a_\varepsilon - b_\varepsilon | < c\varepsilon$ for all $\varepsilon > 0$.
Hence $a_\varepsilon, b_\varepsilon \to 0$ as $\varepsilon \to 0$.
However, $1 + \Gamma$ is a neighborhood of $0$ as $-1 \in \Gamma$.
Hence $b_\varepsilon$ is contained in $1 + \Gamma$ when $\varepsilon$ is sufficiently small, contradiction.
Suppose (2).
Then we have $a_\varepsilon \in 1 + \Gamma$ and $b_\varepsilon \notin 1 + \Gamma$, so $a_\varepsilon - 1$ and $b_\varepsilon - 1$ lie in distinct cosets of $\Gamma$.
Again applying the claim $|a_\varepsilon - 1|, |b_\varepsilon - 1|$ are both bounded above by $c| \varepsilon |$, hence $a_\varepsilon, b_\varepsilon \to 1$ as $\varepsilon \to 0$.
Again this is a contradiction as $\Gamma$ is a neighborhood of $1$.
\end{proof}

\bibliographystyle{amsalpha}
\bibliography{refs}

\end{document}